\documentclass[a4paper,11pt,reqno]{article}
\usepackage{relsize}
\usepackage{cite}
\usepackage{color}
\usepackage{mathtools}
\usepackage[dvipsnames]{xcolor}

\usepackage{hyperref, enumitem}

\hypersetup{
  colorlinks   = true, %Colours links instead of ugly boxes
  urlcolor     = blue, %Colour for external hyperlinks
  linkcolor    = Purple, %Colour of internal links%
  citecolor   = red %Colour of citations
}
\usepackage{amsmath,amsthm,amssymb}

\usepackage[margin=2.5cm]{geometry}

\usepackage{graphicx}
\usepackage[T1]{fontenc}
\usepackage{authblk}
\usepackage[english]{babel}
\usepackage{stackrel}

\theoremstyle{definition}
\newtheorem{theorem}{Theorem}[section]

\newtheorem{proposition}[theorem]{Proposition}
\newtheorem{corollary}[theorem]{Corollary}
\newtheorem{lemma}[theorem]{Lemma}
\newtheorem{definition}[theorem]{Definition}

\newtheorem{remark}[theorem]{Remark}

\newcommand{\C}{\mathcal{C}}

\newcommand{\F}{\mathbb{F}}
\newcommand{\R}{\mathbb{R}}

\newcommand{\N}{\mathbb{N}}

\newcommand{\wt}{\mathrm{wt}}

\newcommand{\M}{\mathcal M}
\newcommand{\mA}{\mathcal A}
\newcommand{\B}{\mathcal B}

\newcommand{\Fq}{\mathbb{F}_q}

\DeclareMathOperator{\PG}{PG}

\title{Strong blocking sets and minimal codes from expander graphs}

\author[1]{Noga Alon\thanks{NA is partially supported  by NSF grant DMS-2154082 and BSF grant 2018267.}}
\affil[1]{Department of Mathematics, Princeton University, United States of America}

\author[2]{Anurag Bishnoi} 

\affil[2]{Delft Institute of Applied Mathematics, Delft University of Technology, Netherlands}

\author[3]{Shagnik Das\thanks{SD is supported by Taiwan NSTC grant 111-2115-M-002-009-MY2.}}

\affil[3]{Department of Mathematics, National Taiwan University, Taiwan}

\author[4]{Alessandro Neri\thanks{AN is supported by the Research Foundation - Flanders (FWO) grant 12ZZB23N.}}

\affil[4]{Department of Mathematics: Analysis, Logic and Discrete Mathematics, Ghent University, Belgium}

\begin{document}
\maketitle

\begin{abstract}
A strong blocking set in a finite projective space is a set of points that intersects each hyperplane in a spanning set. We provide a new graph theoretic construction of such sets: combining constant-degree expanders with asymptotically good codes, we explicitly construct strong blocking sets in the $(k-1)$-dimensional projective space over $\mathbb{F}_q$ that have size $O( q k )$. 
Since strong blocking sets have recently been shown to be equivalent to minimal linear codes, our construction gives the first explicit construction of $\mathbb{F}_q$-linear minimal codes of length $n$ and dimension $k$, for every prime power $q$, for which $n = O (q k)$. This solves one of the main open problems on minimal codes. 
\end{abstract}

\section{Introduction}
A blocking set in a finite projective or affine space is a set of points that intersects every hyperplane.  The study of these objects is a classic topic in finite geometry~\cite{brouwer1978blocking, blokhuis2011blocking}, with many applications in coding theory, combinatorics and computer science. One can strengthen this notion to that of a \emph{strong blocking set} by requiring that the intersection with every hyperplane is not just nonempty, but forms a spanning set for that hyperplane. For example, in a projective plane, the set of all points on a single line is a blocking set, while the set of all points on three non-concurrent lines is a strong blocking set. This special kind of blocking set has also appeared in the literature under the names of generating sets~\cite{fancsali2014lines, heger2015search} and cutting blocking sets~\cite{bonini2021minimal,alfarano2022geometric,bartoli2022cutting}, but in this paper we follow the nomenclature of~\cite{davydov2011linear, heger2021short}.
 
 Strong blocking sets have recently been shown to be in one-to-one correspondence with minimal codes~\cite{alfarano2022geometric, tang2021full}, a notion from coding theory. 
 A linear code is simply a vector subspace of $\mathbb{F}_q^n$. 
 A codeword in a linear code is called minimal if its support does not contain the support of any other codeword apart from its scalar multiples. 
 Minimal codewords in a linear code have been studied for their applications in decoding algorithms~\cite{hwang1979decoding} and cryptography~\cite{massey1993minimal, chabanne2013towards}. 
 Determining the set of minimal codewords in a linear code is a difficult task that has only been achieved for a few families of linear codes, and this has led to the study of \emph{minimal codes}, where \textit{every} non-zero codeword is minimal (see, for example,~\cite{chabanne2013towards}). 
 Recently, minimal codes have also been linked to perfect hash families, and in particular the trifference problem~\cite{Bishnoi2023trifference}, which have important applications in computer science (see for example~\cite{wang2001explicit} and the references therein). 
 
The main problem is to construct minimal codes of a given dimension $k$ of the shortest possible length $n$. It is known that any strong blocking set in the $(k - 1)$-dimensional projective space obtained from $\mathbb{F}_q^k$, denoted by $\mathrm{PG}(k - 1, q)$, must have size at least $(q + 1)(k - 1)$~\cite{alfarano2022three}.
Using the aforementioned connection, this implies that any minimal code of length $n$ and dimension $k$ over $\mathbb{F}_q$ must satisfy $n \geq (q + 1)(k - 1)$. 
Therefore, we would like to construct minimal codes of length close to this lower bound. 
Providing additional motivation for this problem is the fact that minimal codes whose length $n$ is at most linear in $k$ (for a fixed $q$) give rise to asymptotically good error-correcting codes~\cite{alfarano2022geometric}. 
While it is easy to show the existence of such short minimal codes using the probabilistic method, it is a challenging and central open problem to give good explicit constructions~\cite{cohen2013minimal}. 
Many constructions of minimal codes have appeared in the last few years~\cite{alfarano2022geometric, bartoli2019minimal, cohen2013minimal, ding2015linear, fancsali2014lines}, but their lengths remain considerably larger than the theoretical lower bound. 
 
Over the binary field, minimal codes are equivalent to linear intersecting codes~\cite{cohen1985linear, cohen1994intersecting}, which are codes with the property that the supports of any two non-zero codewords have non-empty intersection, but over larger fields it is a more restrictive notion than intersecting codes~\cite{cohen2013minimal}. 
 By this equivalence, we already have an explicit construction of minimal codes for $q = 2$ with $n$ a linear function of $k$~\cite[Theorem 2.3]{cohen1994intersecting}. 
 Bartoli and Borello~\cite[Corollary 3.3]{bartoli2023small} recently gave an explicit construction of strong blocking sets with size linear in the dimension, for any fixed $q \geq 3$, but the dependency on $q$ in their construction is not linear: they proved that for every prime power $q$, there exists an infinite sequence of dimensions $k$ for which they give an explicit construction of a strong blocking set in the projective space $\mathrm{PG}(k - 1, q)$ of length roughly $q^4 k/4$. The same construction appears in an earlier work of Cohen, Mesnager and Randriam~\cite{cohen2016yet}, and the main idea is to concatenate algebraic geometric codes with the simplex code. The argument used in~\cite{bartoli2023small} also has the limitation that it can at best give an explicit construction of size approximately $q^2 k$.

In this paper, we provide a novel graph-theoretic construction combining linear codes with graphs to produce minimal codes. By using explicit constructions of asymptotically good linear codes and constant-degree expander graphs, we then obtain, for some absolute constant $c$, the first explicit construction of strong blocking sets of size $cqk$ in the projective space $\mathrm{PG}(k-1,q)$, and thus also of a $k$-dimensional minimal code over $\Fq$ of length at most $cqk$. By optimising the constant, we can show that our construction improves the previous best explicit constructions for every fixed $q \ge 7$.

 There is a rich history of using expander graphs in the construction of asymptotically good linear codes~\cite{alon1992construction, sipser1996expander, tanner1981recursive}, and we extend this line of research by showing that they can also be used to construct minimal codes. Central to our construction is the notion of vertex integrity of a graph, which measures how many vertices need to be removed from a graph to break it into small components (see Section~\ref{sec:integrity} for a precise definition and references), and we prove a new lower bound on this parameter for $d$-regular graphs in terms of their eigenvalues. This in particular implies that the vertex integrity of constant-degree $n$-vertex expanders is linear in $n$. 
 
 Finite geometry has often been used to give extremal, or near extremal, constructions of graphs with respect to some property (for example, in Tur\'{a}n and Ramsey problems). 
 We show that the other direction can also be fruitful, as in our construction we use extremal graphs to pick a subset of lines whose union has desirable intersection properties with hyperplanes in a finite projective space. This novel construction has also been used to give explicit constructions of certain affine blocking sets~\cite{Bishnoi2023trifference}, and we expect that it will lead to many new results in finite geometry. 

\subsection{Outline}

In Section~\ref{sec:prelim}, we give the necessary background on codes, blocking sets and expander graphs. We introduce the integrity of a graph in Section~\ref{sec:integrity}, and prove our lower bound for regular graphs. In Section~\ref{sec:construction}, we describe our new explicit construction, proving the main result of this paper. In Section~\ref{sec:AGcodes}, we optimize the size of our construction by using algebraic-geometric codes, almost Ramanujan graphs, and field reduction. Finally, in Section~\ref{sec:conclusion}, we summarize our results and highlight some possible directions for further research.

\section{Preliminaries}
\label{sec:prelim}

In this section we recall some basic notions and preliminary results from coding theory, with a focus on minimal linear codes and on how they can be viewed geometrically. We also recall the notion of expander graphs and some explicit constructions. For the rest of this paper, we shall assume that $q$ is a prime power.

\subsection{Error-correcting codes and minimal codes}

Let us fix $\Fq$ to be the finite field with $q$ elements and let $n \in \mathbb N$. 

\begin{definition}
 The \textbf{(Hamming) support} of a vector $v\in \Fq^n$ is the set 
 $$ \sigma(v):=\{ i \, : \, v_i\neq 0\} \subseteq [n].$$
 The \textbf{(Hamming) weight} of $v$ is
 $$\wt(v):=|\sigma(v)|.$$
\end{definition}

The Hamming weight induces a  distance on $\Fq^n$, given by $d(u,v):=\wt(u-v)$. This is known as the \textbf{Hamming distance} and it is fundamental in the theory of error-correcting codes. 

\begin{definition}
An $[n,k,d]_q$ code $\C$ is a $k$-dimensional subspace of $\Fq^n$, and
$$d \coloneqq \min\{\wt(v) \, : \, v \in \C\setminus\{0\}\}$$
is called the \textbf{minimum distance} of $\C$. The elements of $\C$ are called \textbf{codewords}. Moreover, a \textbf{generator matrix} for $\C$ is a matrix $G\in \Fq^{k\times n}$ such that 
$$\C=\{uG\,:\, u \in \Fq^k\};$$
that is, the rows of $G$ span $\C$.
\end{definition}

\begin{definition}
Let $\{n_i \}_{i \geq 1}$ be an increasing sequence of lengths and suppose that there exist sequences $\{k_i\}_{i \geq 1}$ and $\{d_i\}_{i \geq 1}$ such that
for all $i \geq 1$ there is an $[n_i, k_i, d_i]_q$ code $\C_i$. 
Then the sequence $\{\C_i\}_{i \geq 1}$ is called an \textbf{$(R,\delta)_q$-family of codes}, where 
the \textbf{rate} of this family is defined as 
\[R \coloneqq \liminf_{i \rightarrow \infty} \frac{k_i}{n_i},\]
and the \textbf{relative distance} is defined as 
\[\delta \coloneqq \liminf_{i \rightarrow \infty} \frac{d_i}{n_i}.\]
\end{definition}

One of the central problems on error-correcting codes is to understand the trade-off between the rate and the relative distance of codes. 
A family of codes for which $R > 0$ and $\delta > 0$, is known as an \textit{asymptotically good code}. 
An easy probablistic argument shows the existence of such codes for every $\delta \in [0, 1 - 1/q)$ and $R = 1 - H_q(\delta)$, where
\[H_q(x) \coloneqq x \log_q(q - 1) - x \log_q(x)  - (1 - x) \log_q(1 - x),\]
is the $q$-ary entropy function, defined on the domain $0 \leq x \leq 1 - 1/q$. 
This is known as the Gilbert-Varshamov bound. 
The first explicit construction of asymptotically good codes was given by  Justesen \cite{justesen1972class}, who showed that for every $0 < R < 1/2$, there is an explicit family of codes with rate $R$ and relative distance $\delta \geq (1 - 2R) H_q^{-1}\left(\frac{1}{2}\right)$. 
Note that for any prime power $q$, $H_q^{-1} \left( \frac{1}{2} \right) \geq H_2^{-1} \left( \frac{1}{2} \right) > 0.11$, and thus there are absolute constants $R, \delta > 0$, not depending on $q$, for which we have an explicit construction of a family of codes with rate $R$ and relative distance $\delta$. 
Improving the values of the rate $R$ and the relative distance $\delta$ for which there is an explicit construction, and reducing the computational complexity of these constructions, has been an active area of research in coding theory since the 1970s (see for example \cite{alon1992construction, ta2017explicit, tsfasman1982modular}). 
One of the most significant developments in the area was the use of modular curves to show that, for $q \geq 49$, there are explicit constructions of linear codes over $\mathbb{F}_q$ that are even better than the probabilistic ones (see~\cite{tsfasman2013algebraic, couvreur2021algebraic} for some recent surveys on these constructions).

In this paper, we study a special class of codes called \emph{minimal (linear) codes}. These are codes with interesting features from a combinatorial point of view.

\begin{definition}
 Let $\C$ be an $[n,k,d]_q$ code. A nonzero codeword $v \in \C$ is said to be \textbf{minimal} (in $\C$) if $\sigma(v)$ is minimal with respect to the inclusion in the set 
 $$ \sigma(\C) \coloneqq \{ \sigma(c) \, : \, u \in \C\setminus\{0\}\}.$$
 The code $\C$ is a \textbf{minimal linear code} if all its nonzero codewords are minimal.
\end{definition}

Minimal codewords were first studied by Hwang for decoding purposes \cite{hwang1979decoding}. Later, they were analyzed by Massey in connection with secret sharing schemes \cite{massey1993minimal}. Since then, minimal codewords and minimal codes attracted renewed interest within the coding theory community (see for example \cite{alfarano2022geometric, ashikhmin1998minimal, bartoli2019minimal, cohen2013minimal}). These concepts were further studied from a combinatorial point of view, since they correspond to circuits in the matroid associated to the dual code~\cite{dosa2004maximum}. 
Recently, minimal codes have also been linked to linear trifferent codes \cite{Bishnoi2023trifference}, which are a special case of perfect hash families \cite{wang2001explicit}.

\subsection{Projective systems and strong blocking sets}

In this section we briefly  describe the geometric dual approach to coding theory, where linear codes can be identified with set of points in a suitable projective space. 
For $k > 1$, the finite projective space of dimension $k - 1$ over the finite field $\mathbb{F}_q$ is defined as 
\[\mathrm{PG}(k - 1, q) \coloneqq \left(\mathbb{F}_q^k \setminus \{{0}\} \right) / \sim,\]
where 
$u \sim v$ if and only if $u = \lambda v$ for some non-zero $\lambda \in \mathbb{F}_q$ (in some circles the same object will be denoted by $\mathbb{P}^{k-1}(\mathbb{F}_q)$). 
%Equivalently, it is a geometry whose points, lines, $\ldots$, $m$-spaces, are the $1$-dimensional, $2$-dimensional, $\ldots$, $(m + 1)$-dimensional, vector subspaces of $\mathbb{F}_q^k$. 
The equivalence class that a non-zero vector $v$ belongs to is denoted by $[v]$. 
The $1$-dimensional, $2$-dimensional, $\dots$, $(k - 1)$-dimensional vector subspaces of $\mathbb{F}_q^k$ correspond to the points, lines, $\dots$, hyperplanes of $\mathrm{PG}(k - 1, q)$. 
We denote the span of a subset $S$ of points in a projective space by $\langle S \rangle$ and the dimension $\dim(\langle S \rangle)$ is one less than the vector space dimension of the corresponding vector subspace. 
For example, the span of two distinct points $P, Q$ in a projective space, which we will also denote by $\langle P, Q \rangle$, is a $1$-dimensional projective subspace that we refer to as the line joining $P$ and $Q$. 

\begin{definition}\label{def:projective_system}
A \textbf{projective} $[n,k,d]_q$ \textbf{system} is a (multi)set of $n$ points, $\M \subseteq \PG(k-1,q)$, such that 
$\langle \M\rangle=\PG(k-1, q)$ and 
$$d=n-\max\{|H\cap \M| : H \text{ is a hyperplane}\}.$$
\end{definition}

The term projective system and the notation used come from the correspondence with linear codes. Indeed, a projective $[n,k,d]_q$ {system} is simply a dual interpretation of a nondegenerate $[n,k,d]_q$ code. More precisely, an $[n,k,d]_q$ code $\C$ is \textbf{nondegenerate} if there is no identically zero coordinate in $\C$. In other words, $\C$ is not contained in any principal hyperplane $H_i:=\{v \in \Fq^n \,:\, v_i=0\}$.

The aforementioned correspondence -- up to equivalence -- comes from putting representatives of the points of the projective system as columns of a $k\times n$ matrix, and considering the code generated by (the rows of) this matrix. Vice versa, given a generator matrix of a nondegenerate code, we can obtain the associated projective system by taking the columns of such a matrix as a multiset of points in $\PG(k-1,q)$. 
There is more work required to make this correspondence well-defined, and 
%but it becomes so when we consider $\mathrm{PGL}(k,q)$-equivalence classes of projective $[n,k,d]_q$ systems and the so-called monomial equivalence classes of nondegenerate $[n,k,d]_q$ codes.
%(\textcolor{red} {I think this should either be defined or omitted.}). \ale{What do you think needs to be defined, exactly? The monomial equivalence?}
we refer the reader to~\cite[Theorem 1.1.6]{tsfasman2013algebraic} for a formal treatment of the correspondence. Due to this correspondence, a sequence $\{\mathcal M_i\}_{i\in \N}$ of projective systems is called an \textbf{$(R,\delta)_q$-family of projective systems} if the corresponding family of codes is an $(R,\delta)_q$-family of codes.

We now define the main finite geometric object studied in this paper. 

\begin{definition}
 A set $\mathcal M\subseteq \PG(k-1,q)$ is said to be a \textbf{strong blocking set} if 
 $$\langle H\cap \M\rangle=H,$$
 for every hyperplane $H$ of $\PG(k-1,q)$.
\end{definition}

\begin{remark}
In the vector space notation, a strong blocking set in $\F_q^k$ is a collection of $1$-dimensional vector subspaces that intersects every $(k-1)$-dimensional vector subspace in a spanning set. 
\end{remark}

\begin{theorem}[see \cite{alfarano2022geometric}, \cite{tang2021full}]
 Let $\C$ be a nondegenerate $[n,k,d]_q$ code and let $G=(g_1 \mid \ldots \mid g_n) \in \Fq^{k\times n}$ be any of its generator matrices. The following are equivalent:
 \begin{enumerate}
     \item $\C$ is a minimal code;
     \item $\M=\{[g_1],\ldots,[g_n]\}$ is a strong blocking set in $\PG(k-1,q)$.
 \end{enumerate}
\end{theorem}

The main and most relevant problem -- from both a coding theoretic and geometric point of view -- is the construction of small strong blocking sets, or, equivalently, of short minimal codes. 
The first step is to ask how small a strong blocking set can be. Answers to this question are partial and given by the following results. 
The first one is a general lower bound observed in~\cite{alfarano2022geometric}, proved using the Combinatorial Nullstellensatz (see~\cite{Bishnoi2023trifference, heger2021short} for alternative proofs using the results from~\cite{jamison1977covering, brouwer1978blocking}).

\begin{theorem}
\label{thm:strong_lower_bound}
 For any prime power $q$, every strong blocking set in $\mathrm{PG}(k - 1, q)$ has size at least  $(q+1)(k-1).$
\end{theorem}

Recently this lower bound has been improved by using Delsarte's linear programming bound in coding theory, which is also known as the MRRW bound. 

\begin{theorem}[see   {\cite[Theorem 1.4]{Bishnoi2023trifference}, \cite[Theorem 3.3]{scotti2023lower}}]
    For any prime power $q$, there is a constant $c_q>1$ such that every strong blocking set in $\mathrm{PG}(k - 1, q)$ has size at least 
$(c_q-o(1))(q+1)(k-1)$.
\end{theorem}

We also have the following existence result shown using the probabilistic method that provides the best-known upper bounds.

\begin{theorem}[see \cite{miklos1984linear} for $q=2$ and \cite{Bishnoi2023trifference, ABN2023} for $q>2$]\label{thm:upper_bound}
 The size of the smallest strong blocking set in $\PG(k-1,q)$ is at most
  $$  \begin{cases}\frac{2k-1}{\log_2(4/3)} & \mbox{ if } q=2, \\
  (q + 1)\frac{2k}{\log_q(\frac{q^4}{q^3-q+1})} & \mbox{ otherwise. }
  \end{cases}$$
\end{theorem}

This is an existence result that does not provide any explicit constructions. We now recall some of the most relevant general explicit constructions of small strong blocking sets that are known in the literature. 
\medskip

\noindent \textbf{Rational normal tangents:} Assume that $q\geq 2k-3$ and that $\mathrm{char}(\Fq)>k$. Fancsali and Sziklai~\cite{fancsali2014lines} showed that under these hypothesis, one can take any distinct $2k-3$ points on a rational normal curve, and then take the union of the tangent lines to this curve at those points. The resulting set is a strong blocking set of size $(2k-3)(q+1)$. In the same paper, they also showed how to get rid of the hypothesis on the characteristic of the field, by using what they call the \emph{diverted tangents} method. However, the hypothesis on the field size must be kept, implying that such a construction provides only finitely many strong blocking sets for a given field size. 

\medskip

\noindent \textbf{Tetrahedron:} This construction is probably the most natural one. It is obtained by selecting any $k$ points in $\PG(k-1,q)$  in general position, and then taking the union of the lines spanned by every pair of these points. It works over every field, but its size  $\binom{k}{2}(q-1)+k$ is quadratic in $k$, while we know by Theorem~\ref{thm:upper_bound} about the existence of strong blocking sets whose size is linear in $k$. The tetrahedron was first observed by Davydov, Giulietti, Marcugini and Pambianco~\cite{davydov2011linear} and then rediscovered by several authors.

\medskip

\noindent \textbf{Line subspreads:} This is a slight improvement on the size of the tetrahedron. It works whenever $k=2t$ is even, and it consists of carefully choosing $t^2$ points in $\PG(t-1,q^2)$, and then using the field reduction map to obtain $t^2$ lines in $\PG(k-1,q)$ whose union is a strong blocking set. This construction has size $\frac{k^2}{4}(q+1)$ and was recently pointed out in \cite{alfarano2022three}. 

\medskip
All these constructions are obtained as unions of lines in the projective space.
This is mainly due to the fact that with such a structure it is easy to control their intersections with subspaces. 
In particular, the main feature that these constructions possess is the following property, which is stronger 
than being a strong blocking set.

\begin{definition}
A set 
$\mathcal L$ of lines in a projective space satisfies the \textbf{avoidance property} if there is no codimension-$2$ space meeting every line $\ell \in \mathcal{L}$. 
\end{definition}

The relation between these sets of lines and strong blocking sets is the following observation of Fancsali and Sziklai \cite[Theorem 11]{fancsali2014lines}, whose proof we include for the sake of convenience. 

\begin{theorem}\label{thm:avoidance}
 If a set $\mathcal{L}$ of lines satisfies the avoidance property, then the point-set $\B = \cup_{\ell \in \mathcal{L}} \ell$ is a strong blocking set.
\end{theorem}
\begin{proof}
    Let $\mathcal{L}$ be a set of lines and let $\B = \cup_{\ell \in \mathcal{L}} \ell$. 
    Assume that $\B$ is not a strong blocking set. 
    Then there exists a hyperplane $\Pi$ such that $\B \cap \Pi$ does not span $\Pi$. 
    In particular, $\B \cap \Pi$ is contained in a hyperplane $H$ of $\Pi$.
    
    Since $\Pi$ is a hyperplane, it meets every line of the projective space. Thus, $\ell \cap \Pi \neq \emptyset$ for all $\ell \in \mathcal{L}$, but since $\ell \subseteq \B$ and $\B \cap \Pi \subseteq H$, it follows that $\ell \cap H \neq \emptyset$. That is, $H$ is a codimension-$2$ subspace meeting every line of $\mathcal{L}$, and so $\mathcal{L}$ does not satisfy the avoidance property.
\end{proof}

\begin{remark}
As shown in~\cite[Lemma 13]{fancsali2014lines}, any collection of lines that satisfy the avoidance property must have size at least $k - 1 + \lfloor (k - 1)/2 \rfloor$, thus giving a lower bound of roughly $1.5(q + 1)(k - 1)$ on the smallest possible size of a strong blocking set that can be constructed using such a set of lines. 
\end{remark}

\subsection{Expander graphs}
In our construction, we will make use of explicit constructions of constant-degree expander graphs. Informally, the edges of expander graphs are very well spread out, ensuring that there are many outgoing edges from all vertex subsets that are not too large. We refer the reader to the survey~\cite{hlw2006} for a formal definition and for various applications of expanders. 

Expansion in graphs can be measured by their spectral properties.
Given an $n$-vertex graph $G$, we denote the eigenvalues of its adjacency matrix by $\lambda_1 \geq \lambda_2 \geq \cdots \geq \lambda_n$. These eigenvalues encode a lot of information about the graph; for instance, if $G$ is connected and $d$-regular, then $\lambda_1 = d$ and $\lambda_2 < d$. A graph $G$ is called an $(n, d, \lambda)$-graph if it is a $d$-regular graph on $n$ vertices with $|\lambda_i| \leq \lambda$ for all $i \ge 2$. 
The following lemma is one of the central tools for studying such graphs. 

\begin{lemma}[Expander-Mixing Lemma]
\label{lem:eml}
Let $G$ be an $(n, d, \lambda)$-graph and $S, T$ be two vertex-subsets of $G$. 
Denote by $e(S, T)$ the number of pairs $(x, y) \in S \times T$ such that $xy$ is an edge of $G$. 
Then 
\[\left| e(S, T) - \frac{d |S| |T|}{n} \right| \leq \lambda \sqrt{|S||T| \left(1 - \frac{|S|}{n}\right) \left( 1- \frac{|T|}{n}\right)}.\]
\end{lemma}

A proof of this lemma can be found in~\cite[Lemma~4.15]{vadhan2012pseudorandomness}. 
Note that the error term on the right-hand side is directly proportional to $\lambda$, and so it is natural to try to make this parameter as small as possible. The Alon--Bopanna bound~\cite{nilli1991second} limits how far one can go, and motivates the definition of Ramanujan graphs, which are the ultimate expanders.

\begin{theorem}[Alon-Bopanna] \label{thm:alonbopanna}
    Let $G$ be an $(n, d, \lambda)$-graph. 
    Then $\lambda \geq 2 \sqrt{d - 1} - o(1)$ as $n \rightarrow \infty$. 
\end{theorem}

\begin{definition}
Let $G$ be a $d$-regular graph with the eigenvalues $d=\lambda_1 \geq \lambda_2 \geq \ldots \geq \lambda_n$. If 
$\max\{ |\lambda_i| \,:\, |\lambda_i|<d\}\leq 2\sqrt{d-1}$, then  $G$ is said to be a \textbf{Ramanujan graph}.
\end{definition}

Lubotsky, Phillips and Sarnak~\cite{lubotzky1988ramanujan} and Margulis~\cite{margulis1988} gave explicit constructions of $d$-regular Ramanujan graphs for $d = p + 1$, where $p$ is prime. We denote by $H_d$ the $d$-regular Ramanujan graph constructed by Lubotsky, Phillips and Sarnak. For the convenience of the reader, we briefly describe this construction. 

Fix a prime $p\equiv 1 \pmod 4$. By Jacobi's four square theorem, there exist exactly $p+1$ integer solutions to the equation 
\begin{equation}\label{eq:Jacobi}p=b_1^2+b_2^2+b_3^2+b_4^2, \qquad b_1>0, b_2,b_3,b_4 \equiv 0 \pmod 2\end{equation}
Now, let $r\equiv 1 \pmod 4$ be a distinct prime. To each solution of \eqref{eq:Jacobi} we associate the matrix
\begin{equation}\label{eq:p+1}\begin{pmatrix}
b_1+ib_2 & b_3 + ib_4 \\
-b_3+ib_4 & b_1-ib_2
\end{pmatrix} \in \F_r^{2\times 2},\end{equation}
where $i$ is a square root of $-1$ in $\F_r$.
If $p$ is a quadratic residue modulo $r$, we define $H_{p+1}$ to be the Cayley graph of $\mathrm{PGL}(2,\F_r)$ with the $p+1$ generators given in~\eqref{eq:p+1}, which has $r(r^2-1)$ vertices. If $p$ is not a quadratic residue modulo $r$, then we define $H_{p+1}$ to be the Cayley graph of $\mathrm{PSL}(2,\F_r)$ with the $p+1$ generators given in~\eqref{eq:p+1}, which has $\frac{r(r^2-1)}{2}$ vertices.

It was shown in~\cite{lubotzky1988ramanujan} that the graphs $H_{p+1}$ are Ramanujan graphs, and hence $\lambda_2(H_{p+1})\leq 2\sqrt{p}$. A few years later, this construction was adapted in~\cite{morgenstern1994existence} to produce $(q + 1)$-regular graphs for prime powers $q$. However, both of these constructions have the disadvantage that they only produce $(n,d,\lambda)$-graphs with very restricted choices of $n$ and $d$. If we slightly relax the requirement that $\lambda \leq 2\sqrt{d - 1}$, we can find explicit constructions for every degree $d$ and every large enough $n$. 

\begin{theorem}[{see \cite[Theorem 1.3]{alon2021explicit}}]\label{thm:epsilon_Ram}
    For every positive integer $d$, and every $\varepsilon > 0$, there is an $n_0(d, \varepsilon)$ such that, for all $n \geq n_0(d, \varepsilon)$ with $nd$ even, there is an explicit construction of an $(n, d,\lambda)$-graph $G_{n,d}^\varepsilon$ with $\lambda \le 2\sqrt{d-1}+\varepsilon$
\end{theorem}

\section{Integrity of a graph}
\label{sec:integrity}
Crucial to our work is the following graph parameter, known as the (vertex) \textit{integrity} of a graph, which was originally introduced in late 1980s as a measure of the robustness of a network under vertex deletion~\cite{bagga1992survey, barefoot1987vulnerability}.

\begin{definition}
Let $G = (V, E)$ be a simple connected graph. 
For any subgraph $H$, let $\kappa(G)$ denote the largest size of a connected component in $H$. The \textbf{integrity} of $G$ is the integer
$$\iota(G):=\min_{S\subseteq V} \left(|S|+\kappa(G - S)\right).$$
\end{definition}

It is a challenging problem to determine the integrity of graphs precisely, or even asymptotically (see~\cite{bagga1992survey} for an old survey and~\cite{balogh2018two, benko2009asymptotic} for some recent bounds on different families of graphs). 
We prove a new lower bound on the vertex integrity of $(n, d, \lambda)$-graphs. 
First, we introduce another graph parameter and show that it is closely related to the integrity of a graph.

\begin{definition}\label{def:biindependentsets}
    For a graph $G$, let $z(G)$ denote the largest integer $z$ such that there are two disjoint sets of vertices in $G$, each of size $z$, with no edge between them.
\end{definition}

\begin{proposition}\label{prop:z(G)}
    For every graph $G = (V, E)$ on $n$ vertices, 
    \[n - 2 z(G) \leq \iota(G) \leq n - z(G).\]
\end{proposition}
\begin{proof}
    For the upper bound, let $A, B$ be two disjoint sets of size $z$ with no edges between them. 
    Put $S = V - (A \cup B)$. 
    Then any connected component in $G - S$ is either contained in $A$ or in $B$, and thus has size at most $z$. 
    Therefore $\iota(G) \leq |S| + z = (n - 2z) + z = n - z$. 

    We now prove the lower bound. 
    Let $z = z(G)$ and let $S$ be a subset of size $\sigma$ such that the maximum size of a connected component in $G - S$ is $\kappa$, with $\sigma + \kappa = \iota(G)$. 
Let $C_1, \dots, C_t$ be the connected components in $G - S$ of sizes $\kappa = c_1 \geq  \cdots \geq c_t$. 
Note that $n - \iota(G) = \sum_{i = 2}^t c_i$, and thus it suffices to upper bound this sum by $2z$. 
Also note that there are no edges between $C_i$ and $C_j$ for any $i \neq j$. 
If $c_1 \geq z + 1$, then by the maximality of $z$ the size of $C_2 \cup \cdots \cup C_t$ is at most $z$, and we are done.
Therefore we have $c_1 \leq z$ and, for the sake of contradiction, we assume that $\sum_{i = 2}^t c_i \geq 2z + 1$. 
Let $2 \le s \le t$ be the largest index $s$ for which $c_s + \cdots + c_t \ge z + 1$. 
Since $c_s \leq c_1$, it follows that 
$c_s + \cdots + c_t \leq z + c_1$. 
Therefore, $c_2 + \cdots + c_{s-1} \geq 2z + 1 - (z + c_1) = z + 1 - c_1$. 
Let $X = C_1 \cup \cdots \cup C_{s-1}$ and $Y = C_s \cup \cdots \cup C_t$. 
Then both $X$ and $Y$ have size at least $z + 1$, which is a contradiction since they do not have any edges between them.     
\end{proof}

\begin{corollary}\label{cor:alon}
For any $(n, d, \lambda)$-graph $G$, we have 
$\iota(G) \geq \left( \frac{d - \lambda}{d + \lambda} \right)n.$
\end{corollary}
\begin{proof}
 Let $z(G)$ be as in Definition~\ref{def:biindependentsets}. For any two sets $S, T$ of vertices with $e(S, T) = 0$ and $|S| = |T| = z(G)$, Lemma~\ref{lem:eml} implies that 
\[z(G) \leq \frac{\lambda n}{d + \lambda}.\] 
Applying the lower bound $\iota(G) \geq n - 2z(G)$ from Proposition~\ref{prop:z(G)} gives $\iota(G) \geq n - 2\frac{\lambda}{d + \lambda} n = \frac{d - \lambda}{d + \lambda} n$.
\end{proof}

\begin{remark}
    A lower bound on the integrity of cubic graphs was proved in~\cite[Theorem 8]{vince2004integrity}. The argument there, along with Cheeger's inequality~\cite[Theorem 2.4]{hlw2006}, can be used to prove the weaker bound of $\iota(G) \geq n \min \{1/2, (d - \lambda)/(3d - \lambda)\}$.
\end{remark}

\begin{remark}
When applied to $d$-regular Ramanujan graphs, Corollary~\ref{cor:alon} yields a lower bound of $\iota(G) = \left(1 - O \left( d^{-1/2} \right) \right)n$. In Appendix~\ref{app:integrity}, we show that the largest possible integrity of $n$-vertex graphs with average degree at most $d$ is in fact of the form $\iota(G) = \left( 1 - \Theta \left( d^{-1} \log d  \right) \right) n$.
\end{remark}

%\red{I wasn't sure we needed to include the alternate proof of the worse bound.}

% \begin{remark}
%     Another short proof that for any $(n,d,\lambda)$-graph
% $G$, $\iota(G) \geq \left(1-\frac{2\lambda}{d}\right)n$ is as follows. Let $S \subseteq V$ satisfy $\iota(G)=|S|+\kappa(G-S)$ and suppose
% $|S|=(1-x)n$. Then by the result proved in \cite{alon1988explicit} the average degree
% in the induced subgraph of $G$ on the complement of $S$ is at least
% $xd-\lambda$. Therefore there is a connected 
% component in this induced subgraph whose average degree is at least that much.
% If the size of this component is $yn$ then using, again, \cite{alon1988explicit}
% it follows that the average degree in this component is at most
% $yd+\lambda$. Therefore $yd+\lambda \geq xd-\lambda$ implying that
% $y \geq x- \frac{2\lambda}{d}$. Thus
% \[
% \iota(G) \geq (1-x)n +yn \geq \left(1-\frac{2\lambda}{d}\right)n.
% \]
% \end{remark}

\section{Constructing Strong Blocking Sets from Graphs}\label{sec:construction}
In this section, we will provide a new general construction inspired by the tetrahedron (see Section~\ref{sec:prelim}). We will use the data from a projective $[n,k,d]_q$ system and a graph on $n$ vertices in order to construct a set of lines with the avoidance property, whose union, in light of Theorem~\ref{thm:avoidance}, forms a strong blocking set.

\begin{definition}
Let $\M=\{P_1,\ldots,P_n\}$ be a set of $n$ points in $\mathrm{PG}(k - 1, q)$ and let  $G=(\M,E)$ be a  graph with vertex set equal to $\M$. We define 
the set of lines
$$\mathcal{L}(\M, G) \coloneqq \{ \langle P_i,P_j\rangle : P_iP_j \in E\}$$
and the set of points
$$\B(\M,G) \coloneqq \bigcup_{\ell \in \mathcal{L}(\M, G)} \ell $$
\end{definition}

\begin{remark} \label{rem:blockingsetsize}
    The size of $\B(\M, G)$ is at most $n + (q - 1)|E|$, since there are $|E|$ lines, each of which contains at most $q - 1$ points not in $\M$. 
\end{remark}

The following result lies at the heart of our construction as it gives a sufficient condition for the line-set $\mathcal{L}(\M, G)$ to satisfy the avoidance property.

\begin{proposition}\label{prop:mainresult}
 Let $\M=\{P_1,\ldots,P_n\}$ be a set of points in $\mathrm{PG}(k-1,q)$ and let $G=(\M,E)$ be a graph whose set of vertices is $\M$. 
 If for every $S\subseteq \M$ there exists a connected component $C$ in $G - S$ such that 
 $$\langle S\cup C\rangle=\PG(k-1,q),$$
 then the set $\mathcal{L}(\M, G) = \{\langle P_i, P_j \rangle : P_iP_j \in E\}$ satisfies the avoidance property; that is, no codimension-$2$ subspace of $\mathrm{PG}(k - 1, q)$ meets every line of  $\mathcal{L}(\M, G)$. 
\end{proposition}

\begin{proof}
Say $G$ satisfies the property and, for the sake of contradiction, let $H$ be a codimension-$2$ subspace that meets every line in $\mathcal{L} = \mathcal{L}(\M, G)$. 
Let $S = H \cap \M$ and let $C$ be a connected component of $G-S$ such that $S$ and $C$ together span the whole space. 

For every edge $e = P_iP_j$ whose endpoints $P_i, P_j$ lie in $\M \setminus S$, there is a corresponding line $\ell_{ij} = \langle P_i, P_j \rangle \in \mathcal{L}$, which by our assumption intersects $H$. 
Since $P_i, P_j \not\in H$, there must be a unique point $Q_{ij} \in \ell_{ij} \cap H$. Thus, writing $\mathcal{Q}$ for the set $\left\{ Q_{ij} : P_i P_j \in E, P_i, P_j \notin S \right\}$, we have $S \cup \mathcal{Q} \subseteq H$.

Now observe that for an edge $P_i P_j \in E$ with endpoints $P_i, P_j \notin S$, if a subspace contains both $P_i$ and $Q_{ij}$, then it must also contain $P_j$, which lies on the line spanned by $P_i$ and $Q_{ij}$. Fixing some point $P_r$ in the component $C \subseteq \M \setminus S$, since every point in $C$ is connected by a path to $P_r$, the previous observation implies that any subspace containing $\mathcal{Q} \cup \{ P_r \}$ must contain all of $C$. Hence,
\[ \langle H \cup \{ P_r \} \rangle \supseteq \langle S \cup \mathcal{Q} \cup \{ P_r \} \rangle \supseteq \langle S \cup C \rangle = \mathrm{PG}(k-1,q). \]
This is a contradiction, as $H$ is a codimension-$2$ subspace, and thus $\langle H \cup \{P_r\} \rangle$ has codimension at least $1$.
\end{proof}

Proposition~\ref{prop:mainresult} provides a general method of constructing strong blocking sets by combining a graph $G$ with a set $\M$ of points in a projective space. However, the construction requires nontrivial interplay between $G$ and $\M$ and their local properties, and it seems quite difficult to design them simultaneously. For this reason, we will simplify the approach by assuming the worst-case global parameters.

\begin{lemma}\label{lem:strongBlocking_S(G)}
 Let $\M$ be a projective $[n,k,d]_q$ system and let $G=(\M,E)$ be a graph of integrity $\iota(G) \geq n - d + 1$.
 Then $\mathcal{L}(\M,G)$ satisfies the avoidance property, and thus $\B(\M,G)$ is a strong blocking set in $\mathrm{PG}(k - 1, q)$ of size at most $n + (q - 1)|E|$. 
\end{lemma}

\begin{proof}
 Let $S$ be an arbitrary subset of $\M$. 
 Since $\iota(G) \geq n - d + 1$, there exists a connected component $C$ in $G$ such that $|S| + |C| \geq n - d + 1$. 
 From the definition of projective systems (see Section~\ref{sec:prelim}), it follows that every hyperplane meets $\M$ in at most $n - d$ points.
 Therefore, $S \cup C \subseteq \M$ is not contained in any hyperplane of $\mathrm{PG}(k - 1, q)$, thus implying $\langle S \cup C \rangle = \mathrm{PG}(k - 1, q)$. 
 From Proposition~\ref{prop:mainresult}, we conclude that $\mathcal{L}(\M, G)$ satisfies the avoidance property and thus, by Theorem~\ref{thm:avoidance}, $\B(\M, G)$ is a strong blocking set. As per Remark~\ref{rem:blockingsetsize}, $|\B(\M, G)| \leq n + (q - 1)|E|$. 
\end{proof}

We now prove the main result of our paper by giving an explicit construction of strong blocking sets in $\mathrm{PG}(k - 1, q)$ with size linear in $q k$. 

\begin{theorem}\label{thm:abs_const}
    There is an absolute constant $c$ such that for every prime power $q$, there exists an explicit construction of strong blocking sets of size at most $c q k_i$ in $\mathrm{PG}(k_i - 1, q)$, for some infinite increasing sequence $\{k_i\}_{i \in \mathbb{N}}$.
\end{theorem}
\begin{proof}
    Let $R$ be any constant satisfying $0 < R < 1/2$ and let $\delta = 0.11(1 - 2R)$. 
    Let $\M_i$ be the projective $[n_i, k_i, d_i]_q$ systems given by the Justesen construction~\cite{justesen1972class}, which exist for an infinite increasing sequence $\{k_i\}_{i \in \mathbb{N}}$.
    Then $\lim_{i \rightarrow \infty} k_i/n_i = R$ and $\lim_{i \rightarrow \infty} d_i/n_i \geq (1 - 2R)H_q^{-1}(1/2) > \delta$. 
    Therefore, there exists an $i_0$, which we may assume to be sufficiently large for all subsequent calculations, such that for all $i \geq i_0$, we have $d_i/n_i \geq \delta$ and $k_i/n_i \geq R/2$.
    Let $\{G_i\}_{i \geq i_0}$ be an explicit family of $(n_i, d, \lambda)$-graphs, where $d$ and $\lambda$ are positive constants for which $(d - \lambda)/(d + \lambda) \geq 1 - \delta + 1/n_i$. 
    From Theorem~\ref{thm:epsilon_Ram}, it follows that such an explicit construction of graphs is always possible.
    By Corollary~\ref{cor:alon}, we have $\iota(G_i) \geq (1 - \delta)n_i + 1 \geq n_i - d_i + 1$.
    Therefore, by Lemma~\ref{lem:strongBlocking_S(G)}, $\B(\M_i, G_i)$ is a strong blocking set in $\mathrm{PG}(k_i - 1, q)$ of size at most \[n_i + (q - 1) \frac{d n_i}{2} < \frac{d}{2}qn_i \leq \frac{d}{R} q k_i.\]
    This concludes the proof with $c = \frac{d}{R}$. 
\end{proof}

\section{Strong blocking sets from expander graphs and AG codes} \label{sec:AGcodes}
Using the construction of Theorem~\ref{thm:abs_const}, the best constant $c$ that we get is quite large; for the optimal choice of $R$, it is approximately $c \simeq 8276$. However, we can reduce it substantially by replacing the Justesen codes with some families of AG codes and -- depending on the field -- by using field reduction. In this section we optimize the value of the constant $c$ in our construction for all values of $q$. 

To this end, we use the asymptotically good Algebraic-Geometry (AG) codes, explicit constructions of which can be found in~\cite{tsfasman1982modular,garcia1995tower,garcia1996asymptotic}. In particular, for every square prime power $q$ and $R, \delta > 0$ satisfying $R + \delta \ge 1 - ( \sqrt{q} - 1)^{-1}$, we can construct an $(R, \delta)_q$-family of $[n_i, k_i, d_i]_q$ codes for some increasing sequences $\{n_i\}_{i \in \N}, \{k_i\}_{i \in \N}$ and $\{d_i\}_{i \in \N}$.

\begin{definition} \label{def:AGprojectivesystem}
	Given a square prime power $q$, for every $R \in (0,1)$, set $\delta = 1 - R - (\sqrt{q} - 1)^{-1}$. Given the $(R,\delta)_q$-family of $[n_i,k_i,d_i]_q$ codes described above, we denote by $\{\mA_{n_i,R}\}_{i \in \N}$ the associated $(R, 1 - R - (\sqrt{q} - 1)^{-1})_q$-family of projective $[n_i, k_i, d_i]_q$ systems.
\end{definition}

With this notation in place, we can now proceed to describe our improved constructions.

\subsection{Quadratic Fields}

We start with a simple result, obtained by combining Lemma~\ref{lem:strongBlocking_S(G)} and Corollary~\ref{cor:alon} with the explicit construction of expander graphs given by Theorem~\ref{thm:epsilon_Ram}. 

\begin{theorem}\label{thm:Ramanujan_AG}
Let $d\ge 3$, let $q>4$ be a square prime power such that $(\sqrt{q}-1)^{-1}<\frac{d-2\sqrt{d-1}}{d+2\sqrt{d-1}}$, and let $\varepsilon >0$. Then there is an increasing sequence $\{k_i\}_{i \in \N}$ for which we can explicitly construct strong blocking sets in $\PG(k_i-1,q)$
of size at most
  $$ \left(\frac{d(d+2\sqrt{d-1})(\sqrt{q}-1)}{2 \left( d(\sqrt{q}-2)-2\sqrt{q(d-1)} \right)}+\varepsilon \right)k_iq.$$
\end{theorem}

Before we proceed with the proof, let us explore what this result implies about that constant in the bound on the size of strong blocking sets in $\mathrm{PG}(k-1,q)$ when $q$ is a square. For each such $q$, we can choose an optimal value for $d$ to minimize the bound. This amounts to finding the minimum values of the function 
\[ F_q(d) = \frac{d ( d + 2 \sqrt{d-1} )( \sqrt{q} - 1) }{ 2 \left( d ( \sqrt{q} - 2) - 2 \sqrt{q(d-1)} \right)}, \]
where we can extend the domain to $\R_{>2}$. To simplify the calculations, we can make the substitution $y = \sqrt{d - 1}$, and then find the local extrema by setting the derivative equal to zero. This amounts to finding the zeros of the polynomial
\[ \psi_q(y) = \sqrt{q} (y-1)(y^3 - 2y^2 - y - 2) - 2(y^2+1)^2. \]
As $q$ grows, the roots of this polynomial converge to those of $(y-1)(y^3 - 2y^2 - y - 2)$, and the unique root in our domain of interest ($y > 1$) is
\[ y_0 = \frac13 \left( 2 + (44 - 3 \sqrt{177})^{\tfrac13} + (44 + 3 \sqrt{177})^{\tfrac13} \right). \]
Hence, for large values of $q$, $F_q(d)$ will be minimized for 
\[ d \approx d_0 = 1 + y_0^2 = 3 + \frac13 (459 - 12 \sqrt{177})^{\tfrac13} + \frac13 (459 + 12 \sqrt{177})^{\tfrac13} \approx 8.0701, \]
and so one should take $d = 8$ or $9$. It is straightforward to verify that, for large enough $q$, we have $F_q(8) < F_q(9)$, and hence, as $q$ tends to infinity, the optimal constant this construction provides is
\[ \lim_{q \to \infty} F_q(8) = \frac49 ( 23 + 8 \sqrt{7} ) \approx 19.63, \]
a very significant saving compared to the construction from the previous section. For smaller values of $q$, we can compute the optimal choice of $d$ and the corresponding constant, and these are given in Table~\ref{table1}.

\begin{table}[ht]
\label{tab:ram_AG}
\begin{center}
\begin{tabular}{|c|c|c|}
\hline
    $q$  & $\mathrm{argmin} \; F_q(d)$ & upper bound/$k(q+1)$   \\
    \hline 
    \hline 
    $9$ & $d=85$ &  $292.68 $ \\
    $16$ & $d=37$ & $104.60 $ \\
    $25$ & $d=26$ & $66.86$\\
    $49$ & $d=18$ & $43.91$\\
    $64$ & $ d=16$ & $39.07$\\
    $81$ & $d=15$ & $35.83$\\
    $121$ & $d=13$ & $31.76$ \\
    $169$ & $d=12$ & $29.31$ \\
    $256\leq q \leq 361$ & $d=11$ & $27.06$ \\
    $529\leq q \leq 1024$ & $d=10$ & $24.44$\\
    $1369\leq q \leq 11881$ & $d=9$ & $22.46$ \\
    $q\geq 12769$ & $d=8$ & $20.52$ \\
    \hline
\end{tabular}
\caption{
{For given ranges of square prime powers, this table provides the values of $d$ that minimize the size of the strong blocking sets obtained by Theorem~\ref{thm:Ramanujan_AG}, and upper bounds on the corresponding sizes.}}\label{table1}
\end{center}
\end{table}

Sufficiently motivated, we now prove the theorem.

\begin{proof}[Proof of Theorem~\ref{thm:Ramanujan_AG}]
Let $\varepsilon_1 = \varepsilon_1(d,q) > 0$ be sufficiently small, and set $\lambda = 2 \sqrt{d-1} + \varepsilon_1$. We have $(\sqrt{q} - 1)^{-1} < \frac{d-\lambda}{d+ \lambda} - 2\varepsilon_1$, and set $R = \frac{d-\lambda}{d + \lambda} - \varepsilon_1 - (\sqrt{q} - 1)^{-1}$ and $\delta = 1 - R - (\sqrt{q} - 1)^{-1} = 1 - \frac{d-\lambda}{d+\lambda} + \varepsilon_1$.

Let $\{ \mA_{n_i,R} \}_{i \in \N}$ be the $(R,\delta)_q$-family of projective $[n_i, k_i, d_i]_q$ systems from Definition~\ref{def:AGprojectivesystem}. Theorem~\ref{thm:epsilon_Ram} shows that there is some $M$ for which we obtain an explicit sequence $\{G_i\}_{i > M}$ of $(n_i, d, \lambda)$-graphs. Corollary~\ref{cor:alon} gives
\[ \iota(G_{n_i}) \ge n_i \frac{d-\lambda}{d+\lambda} = n_i \left( 1 - \delta + \varepsilon_1\right). \]

Since $\lim_{i \to \infty} \frac{d_i}{n_i} = \delta$, we have $\iota(G_{n_i}) \ge n_i - d_i + 1$ for sufficiently large $i$. Thus, by Lemma~\ref{lem:strongBlocking_S(G)}, $\B(\mA_{n_i,R}, G_{n_i})$ is a strong blocking set in $\mathrm{PG}(k_i-1,q)$. Since $G_{n_i}$ has $\tfrac12 n_i d$ edges, we have
\[ \left| \B(\mA_{n_i,R},G_{n_i}) \right| \le n_i + (q-1)\frac{n_id}{2} < \frac{n_id}{2}q. \]

Now, since $\lim_{i \to \infty} \frac{k_i}{n_i} = R$, we have $n_i \le \frac{k_i}{R - \varepsilon_1}$ for sufficiently large $i$. Making this substitution, and recalling our choice of $R = \frac{d-\lambda}{d+\lambda} - \varepsilon_1 - (\sqrt{q} - 1)^{-1}$, our upper bound becomes
\[ \frac{n_id}{2}q \le \frac{d}{2(R - \varepsilon_1)}k_1 q = \frac{d(d + \lambda)(\sqrt{q} - 1)}{2 \left( (d-\lambda)(\sqrt{q} - 1) - (d + \lambda) - 2 \varepsilon_1 (d + \lambda)(\sqrt{q} - 1) \right) } k_i q. \]

If we choose $\varepsilon_1$ to be sufficiently small, we obtain the upper bound 
\[ \left| \B( \mA_{n_i,R}, G_{n_i} ) \right| \le \left( \frac{d(d+\lambda)(\sqrt{q} - 1)}{2 \left( (d - \lambda)(\sqrt{q} - 1) - (d+ \lambda) \right)} + \frac{\varepsilon}{2} \right) k_i q. \]

Recalling that $\lambda = 2 \sqrt{d-1} + \varepsilon_1$, we have 
\[  \frac{d(d+\lambda)(\sqrt{q} - 1)}{2 \left( (d - \lambda)(\sqrt{q} - 1) - (d+ \lambda) \right)} = \frac{d(d + 2 \sqrt{d-1} )(\sqrt{q} - 1) + \varepsilon_1 d (\sqrt{q} - 1) }{2 \left( d ( \sqrt{q} - 2) - 2 \sqrt{q(d-1)} - \varepsilon_1 \sqrt{q} \right)}, \]
and the result follows provided $\varepsilon_1$ is suitably small.
\end{proof}

\subsection{Non-Quadratic Fields}

Theorem~\ref{thm:Ramanujan_AG} shows that replacing the Justesen codes with AG codes in our construction can greatly reduce the size of the strong blocking sets we obtain. However, the one drawback is that the construction is only possible over quadratic fields. In this section we show how to use one final trick --- field reduction --- to take a strong blocking set over $\F_{q^2}$ and build from it a strong blocking set over $\F_q$ that is not much larger.

We first recall the \textbf{field reduction map}, which we denote by $\mathcal F_{q,r}$. This map uses the fact that points of $\PG(K-1,q^r)$ are $1$-dimensional $\F_{q^r}$-subspaces of $\F_{q^r}^K$, which in turn can be viewed as $r$-dimensional $\F_q$-subspaces of $\F_{q}^{rK}$. Hence, $\mathcal{F}_{q,r}$ sends points of $\PG(K-1,q^r)$ to $(r-1)$-spaces of $\PG(rK-1,q)$; see~\cite{lavrauw2015field} for a survey on field reduction.

As shown in~\cite{alfarano2022three}, the field reduction map also preserves some key properties related to strong blocking sets. We begin with a definition.

\begin{definition}\label{def:derived}
Let $\mathcal{L} = \{ \ell_1, \ell_2, \hdots, \ell_t\}$ be a collection of lines in $\PG(K-1, q^2)$. We say a set of points $\Lambda \subseteq \PG(K-1,q^2)$ is \textbf{viable for $\mathcal{L}$} if $\Lambda = \cup_{i=1}^{t} \Lambda^{(i)}$, where each $\Lambda^{(i)} = \left\{ \lambda_1^{(i)}, \lambda_2^{(i)}, \lambda_3^{(i)}, \lambda_4^{(i)} \right\} \subseteq \ell_i$ is a set of four points that do not lie on a common $\F_q$-subline of $\ell_i$. 

Given a viable set $\Lambda$, we define the \textbf{derived set} to be the set 
\[ \mathcal{F}_{q,2}(\Lambda) = \left\{ \mathcal F_{q,2}(\lambda_j^{(i)}) : i \in [t], j \in [4] \right\} \]
of lines in $\PG(2K-1, q)$.
\end{definition}

The following result, obtained by combining~\cite[Theorem 4.2]{alfarano2022three} and~\cite[Proposition 4.5]{alfarano2022three}, allows us to turn strong blocking sets in $\PG(k-1,q^2)$ into strong blocking sets in $\PG(2k-1,q)$. This was also highlighted in \cite{ABN2023}, where viable sets are shown to be \textbf{outer strong blocking sets}.

\begin{theorem}[see {\cite{alfarano2022three}}]\label{thm:field_reduction}
    Let $\mathcal{L} = \{ \ell_1, \ldots, \ell_t\}$ be a set of lines in $\PG(K-1,q^2)$ whose union forms a strong blocking set. If $\Lambda$ is a viable set for $\mathcal{L}$, then the union of the lines in the derived set $\mathcal F_{q,2}(\Lambda)$ is a strong blocking set in $\PG(2K-1,q)$. 
    %\red{Is more true? Can we say that if $\mathcal{L}$ has the avoidance property, then so does the derived set, $\mathcal F_{q,2}(\Lambda)$?}\ale{It is. Do you want me to write down the proof? Not sure it is needed.}
\end{theorem}

This field reduction process is especially effective when used on our strong blocking sets $\B(\M,G)$ constructed from graphs, as the points of $\M$ belong to several lines. 

\begin{lemma}\label{lem:strong_from_strongS}
 Let $\M$ be an $[n,K,d]_{q^2}$ projective system and let $G=(\M,E)$ be a graph. Then we can find a viable set $\Lambda$ of size at most $n + 2|E|$ for the associated set of lines $\mathcal{L}(\M,G)$ in $\PG(K-1,q^2)$.
\end{lemma}

\begin{proof}
  Let us enumerate the edges of $G$ as $E = \{ e_1, e_2, \hdots, e_m \}$. If $e_i = P_a P_b$, then the corresponding line $\ell_i \in \mathcal{L}(\M, G)$ is given by $\ell_i = \langle P_a, P_b \rangle$. We then take $\lambda^{(i)}_1 = P_a$ and $\lambda^{(i)}_2 = P_b$, and let $\lambda^{(i)}_3 = Q^{(a,b)}_3$ be an arbitrary third point on $\ell_i$. Since any three points on an $\F_{q^2}$-line define a unique $\F_q$-subline, we can then choose a fourth point $\lambda^{(i)}_4 = Q_4^{(a,b)}$ that avoids this subline.

Thus,
\[ \Lambda = \bigcup_{i=1}^m \left\{ \lambda^{(i)}_1, \lambda^{(i)}_2, \lambda^{(i)}_3, \lambda^{(i)}_4 \right\} = \M \cup \left( \bigcup_{P_a P_b \in E} \left\{ Q^{(a,b)}_3, Q^{(a,b)}_4 \right\} \right) \]
is viable for $\mathcal{L}(\M,G)$, and $|\Lambda| \le n + 2|E|$. 
\end{proof}

We can now apply Lemma~\ref{lem:strong_from_strongS} and Theorem~\ref{thm:field_reduction} to the construction from Theorem~\ref{thm:Ramanujan_AG} in order to build small strong blocking sets even when $q$ is not a square.

\begin{theorem}\label{thm:Ramanujan_AG_FR}
Let $d\ge 3$, let $q>2$ be such that $(q-1)^{-1}<\frac{d-2\sqrt{d-1}}{d+2\sqrt{d-1}}$, and let $\varepsilon >0$. Then there is an increasing sequence $\{k_i\}_{i \in \N}$ for which we can explicitly construct strong blocking sets in $\PG(k_i-1,q)$ of size at most
  \[ \left(\frac{(d+1)(d+2\sqrt{d-1})(q-1)}{2 \left( d(q-2)-2q\sqrt{d-1} \right)}+\varepsilon \right)k_i(q+1). \]
\end{theorem}

\begin{proof}
This proof follows the same lines as that of Theorem~\ref{thm:Ramanujan_AG}, and so we will mainly highlight the changes. As before, we let $\varepsilon_1 = \varepsilon_1(d,q) > 0$ be sufficiently small and set $\lambda = 2 \sqrt{d-1} + \varepsilon_1$. By assumption, $(q - 1)^{-1} < \frac{d- \lambda}{d + \lambda} - 2 \varepsilon_1$, and we set $R = \frac{ d - \lambda }{ d + \lambda } - \varepsilon_1 - ( q - 1)^{-1}$ and $\delta = 1 - R - (q - 1)^{-1} = 1 - \frac{d - \lambda}{d + \lambda} + \varepsilon_1$.

We then take $\left\{ \mA_{n_i,R} \right\}_{i \in \N}$ to be an $(R,\delta)_{q^2}$-family of projective $[n_i, K_i, d_i]_{q^2}$ systems, and $\{ G_i \}_{i > M}$ a sequence of $(n_i, d, \lambda)$-graphs. As before, our choice of parameters ensures that for sufficiently large $i$, we have $\iota(G_{n_i}) \ge n_i - d_i + 1$. Applying Lemma~\ref{lem:strongBlocking_S(G)}, we deduce that the set of lines $\mathcal{L} = \mathcal{L}(\mA_{n_i,R},G)$ has the avoidance property in $\PG(K-1,q^2)$, and hence, by Theorem~\ref{thm:avoidance}, the union $\B(\mA_{n_i,R},G)$ of those lines is a strong blocking set.

By Lemma~\ref{lem:strong_from_strongS}, we can find a set $\Lambda$ that is viable for $\mathcal{L}$ of size at most $|\mA_{n_i,R}| + 2e(G_{n_i}) = n_i(d+1)$. Theorem~\ref{thm:field_reduction} shows that the union of the lines in the derived set $\F_{q,2}(\Lambda)$ is then a strong blocking set in $\PG(2K_i-1,q)$. Since each line in $\PG(2K_i-1,q)$ has $q+1$ points, this strong blocking set has size at most $n_i(d+1)(q+1)$. Since $R = \lim_{i \to \infty} \frac{K_i}{n_i}$, recalling our choices for the parameters $R$ and $\lambda$ and setting $k_i = 2K_i$ then yields the claimed bound.
\end{proof}

As before, we can determine the optimal degree $d$ to use by minimizing the quantity
\[ R_q(d) := \frac{(d+1)(d + 2 \sqrt{d-1} )(q-1)}{2 \left( d(q-2) - 2q \sqrt{d-1} \right)}. \]
It is again advisable to make the substitution $d = 1 + y^2$, following which we find the zeros of the derivative coincide with those of $\phi_q(y) = q(y-1)(y^3 - 2y^2 - y - 4) - 2(y^2 - y + 1)(y^2 + y + 2)$. As $q$ grows, the roots of $\phi_q$ converge to the roots of $(y-1)(y^3 - 2y^2 - y - 4)$. This polynomial has a unique root $y_0$ that is larger than $1$, and this corresponds to
\[ d_0 = 1 + y_0^2 = 3 + (31 - 2 \sqrt{58})^{\tfrac13} + (31 + 2 \sqrt{58})^{\tfrac13} \approx 9.0967. \]

Hence the asymptotically optimal degree must be either $9$ or $10$, and inspection shows $R_q(10) > R_q(9) \to \frac{5}{49} \left( 113 + 72 \sqrt{2} \right) \approx 21.92$ as $q \to \infty$.

For large $q$, then, Theorem~\ref{thm:Ramanujan_AG_FR} yields a larger strong blocking set than Theorem~\ref{thm:Ramanujan_AG}. However, aside from the fact that Theorem~\ref{thm:Ramanujan_AG_FR} works over any field, not just quadratic ones, it also outperforms Theorem~\ref{thm:Ramanujan_AG} for small values of $q$. More precise estimates are given in Table~\ref{table2}, and in comparison to Table~\ref{table1}, we find that the field reduction can lead to significantly smaller constants when $q$ is small.

\begin{table}[ht]
\begin{center}
\begin{tabular}{|c|c|c|}
\hline
     $q$  & $\mathrm{argmin} \; R_q(d)$ & upper  bound/$k(q+1)$   \\
    \hline 
    \hline 
    $3$ & $d=85$ &  $296.12 $ \\
    $4$ & $d=38$ &  $107.35 $ \\
    $5$ & $d=27$ & $ 69.41 $ \\
    $7$ & $d=19$ & $46.32$\\
    $8$ & $d=17$ & $41.45$\\
    $9$ & $ d=16$ & $38.18$\\
    $11$ & $d=14$ & $34.08$\\
    $13$ & $d=13$ & $31.62$ \\
    $16\leq q \leq 19$ & $d=12$ & $29.36$ \\
    $23\leq q \leq 32$ & $d=11$ & $26.73$\\
    $37 \leq q \leq 109$ & $d=10$ & $24.75$ \\
    $q\geq 113$ & $d=9$ & $22.81$ \\
    \hline
\end{tabular}
\caption{For given ranges of prime powers, this table provides the degrees $d$ that minimize the sizes of the strong blocking sets constructed in Theorem~\ref{thm:Ramanujan_AG_FR}, and upper bounds on their sizes. }\label{table2}
\end{center}
\end{table}

\section{Conclusion} \label{sec:conclusion}

In this paper, we describe a general machinery for constructing strong blocking sets in finite projective spaces starting from a graph and a linear code. In particular, taking explicit constructions of constant-degree expanders and asymptotically good linear codes, we provide the first explicit construction of strong blocking sets in $\PG(k-1,q)$ whose size is linear in both $k$ and $q$. As a consequence, as highlighted in Theorem~\ref{thm:abs_const}, this also provides an explicitly constructed family of asymptotically good minimal codes over $\Fq$ of rate at least $(cq)^{-1}$, for some absolute constant $c$. These constructions are based on our new results on the vertex integrity of a graph. Concretely, Corollary~\ref{cor:alon} bounds the vertex integrity of $d$-regular graphs from below by a quantity only depending on their eigenvalues. Finally, in Section \ref{sec:AGcodes} we optimize the constant $c$: we make use of almost Ramanujan graphs and asymptotically good families of AG codes (Theorem~\ref{thm:Ramanujan_AG}), and combine this construction with the field reduction on a viable set of points, obtaining a derived strong blocking set (Theorem~\ref{thm:Ramanujan_AG_FR}).

It must be noted that the strong blocking sets from Theorem~\ref{thm:Ramanujan_AG_FR} are themselves the unions of lines in $\PG(k-1,q)$, and if $q$ is a square, we can again apply Theorem~\ref{thm:field_reduction} to derive strong blocking sets in $\PG(2k-1,\sqrt{q})$. We can then repeat this process further until we reach a field of non-square order. 
Thus, we can start with 
%The following corollary describes the strong blocking sets over $\F_q$ that we can obtain if we start 
a strong blocking set over $\F_{q^{2^r}}$, and then get a strong blocking set over $\mathbb{F}_q$ in $r$ steps. 
The calculations, which we omit, are essentially the same as those in the proof of Theorem~\ref{thm:Ramanujan_AG_FR}. The only exception is that since our intermediate blocking sets will not be coming from graphs, we cannot apply Lemma~\ref{lem:strong_from_strongS} each time to produce very economical viable sets. Instead, since we choose four points from each line to make a viable set, we shall bound the number of lines in each iteration as being at most four times larger than in the previous step. This allows us to derive an
% \begin{corollary}[$r$-th derived strong blocking sets]\label{cor:Ramanujan_AG_2FR}
% Let $d\ge 3$, $r\geq 1$, let $q$ be a prime power such that $(q^{2^r}-1)^{-1}<\frac{d-2\sqrt{d-1}}{d+2\sqrt{d-1}}$, and let $\varepsilon >0$ be arbitrary small. Then, , we can 
explicit construction of strong blocking sets in $\PG(k_i-1,q)$ of size at most
  \[ \left( \frac{2^{r-1}(d+1)(d+2\sqrt{d-1})(q^{2^r}-1)}{d(q^{2^r}-2)-2q^{2^r}\sqrt{d-1}}+\varepsilon \right) k_i(q+1). \]
  for some increasing sequence $\{k_i\}_{i \in N}$, provided that $(q^{2^r}-1)^{-1}<\frac{d-2\sqrt{d-1}}{d+2\sqrt{d-1}}$.

% Note that these repeated derivations allow us to construct strong blocking sets over the binary field $\F_2$ as well. Moreover, one can again find the optimal degree $d$ for each prime power $q$ by minimizing the constant. We can in fact reuse our previous calculations for this, since 
% \[ \frac{2^{r-1} (d+1) (d + 2 \sqrt{d-1})(q^{2^r} - 1)}{d(q^{2^r} - 2) - 2q^{2^r} \sqrt{d-1}} = 2^r R_{q^{2^r}}(d). \]

Performing the optimization reveals that the third derivation ($r = 3$) minimizes the size of the blocking sets for $q = 2$ and that the second derivation ($r = 2$) is optimal for $3 \le q \le 5$. For $q \ge 7$, however, the strong blocking sets from Theorems~\ref{thm:Ramanujan_AG} and~\ref{thm:Ramanujan_AG_FR} are already so small --- they are better than the previous best-known constructions~\cite{bartoli2023small, cohen2016yet} --- that repeated derivations offer no improvement. For the convenience of the reader, we summarize in Table~\ref{table3} the smallest strong blocking sets we obtained using our constructions.

\begin{table}[ht!]
\begin{center}
\begin{tabular}{|c|c|c|}
\hline
     $q$  & Construction & upper  bound/$k(q+1)$   \\
    \hline 
    \hline 
    $2$ & $3$rd derivation &  $118$ \\
    $3$ &  $2$nd derivation & $77$\\
    $4$ & $2$nd derivation & $59$ \\
    $5$ & $2$nd derivation & $54$\\
    $7$ & $1$st derivation & $47$\\
    $8$ & $1$st derivation & $  42$\\
    $9$ & $1$st derivation & $39$\\
    $11$ & $1$st derivation & $35$\\
    $13$ & $1$st derivation & $32$ \\
%    $16\leq q \leq 19$ & $1$st derivation & $30$ \\
 $16$ &  $1$st derivation & $30$\\
$17$ &  $1$st derivation & $29$\\
 $19$ &  $1$st derivation & $28$\\
    $23\leq q \leq 25$ & $1$st derivation & $27$\\
    $27\leq q \leq 32$ & $1$st derivation & $26$\\
    $37 \leq q \leq 49$ & $1$st derivation & $25$ \\   
      $53 \leq q \leq 109$ & $1$st derivation & $24$ \\   
    $113 \leq q \leq 1217 $ & $1$st derivation & $23$ \\
        $1223 \leq q \leq 12763 $ & $1$st derivation & $22$ \\
         $q\geq 12769$, $q$ non-square & $1$st derivation & $22$ \\
    $12769 \le q < 70603$,  $q$ square & original & $21$ \\
    $q > 70603$, $q$ square  & original & $20$\\
    \hline
\end{tabular}
\caption{For given ranges of prime powers $q$, this table provides the 
best upper bound on the size of the constructed
strong blocking sets in finite projective spaces over the finite field $\Fq$, together with an indication of which construction achieve this: Original (Theorem~\ref{thm:Ramanujan_AG}), 1st derivation (Theorem~\ref{thm:Ramanujan_AG_FR}) or $r$th derivation.}
%\red{I'm not sure we should present this table in its current form. First, it's a little confusing to use the \# lines as a metric here, compared to using the constant in the previous tables. Second, we say that the original construction is optimal for $q \ge 1369$, but that only applies to square $q$. So it feels like this information can be more clearly presented in text, with references to Tables~\ref{table1} and~\ref{table2}.}}
\label{table3}
\end{center}
\end{table}

Our construction of strong blocking sets can be used to give explicit constructions of affine blocking sets with respect to codimension-$2$ subspaces \cite{Bishnoi2023trifference}. 
Motivated by this problem, it will be interesting to explore a generalization of our construction to $r$-uniform hypergraphs, which would lead to a construction of affine blocking sets with respect to codimension-$r$ subspaces. 

\bibliographystyle{abbrv}
\bibliography{references}

\appendix

\section{Tight bounds for integrity} \label{app:integrity}

In this appendix, we show that the maximum possible integrity of an $n$-vertex graph of average degree $d$ is $\left( 1 - \Theta \left( \frac{\log d}{d} \right) \right)n$, complementing the bounds given in Section~\ref{sec:integrity}.

\begin{proposition} \label{prop:tightintegrity}
Let $d \ge 2$.
\begin{itemize}
	\item[(i)] If $n \ge 48 d$, the integrity of any $n$-vertex graph $G$ of average degree at most $d$ satisfies 
\[ \iota(G) \le \left( 1 - \frac{\log d}{4 d} \right) n. \]
	\item[(ii)] For all $n \ge d$, there are $n$-vertex graphs $G$ of average degree at most $d$ with
\[ \iota(G) \ge \left(1 - \frac{4 \log d}{d} \right) n. \]
\end{itemize}
\end{proposition}

\begin{proof}
For both parts, we shall appeal to Proposition~\ref{prop:z(G)}, which asserts that $n - 2z(G) \le \iota(G) \le n - z(G)$, where $z(G)$ is the largest $z$ such that $G$ contains two disjoint sets of $z$ vertices that have no edges between them.

\medskip

For part (i), we need to show that $z(G) \ge \frac{n \log d}{4 d}$ for all such graphs $G$. This follows from the K\"ov\'ari--S\'os--Tur\'an Theorem~\cite{KST}, applied to the complement of $G$. For the sake of completeness, though, we provide a simple probabilistic proof.

Let $A$ be a random subset of $V = V(G)$ obtained by selecting each vertex of $G$ independently with probability $p = \frac{\log d}{2 d}$. The size of $A$ is then a binomial random variable, and the Chernoff bound shows that $|A| \ge \frac{n \log d}{4 d}$ with probability at least $1 - e^{- n \log d / (16 d)}$, which is at least $1 - \frac{1}{d^3}$.

We now define $B$ to be the set of all vertices in $V \setminus A$ that have no neighbors in $A$. Note that if a vertex $v$ has degree $d_v$, then $\mathbb{P}(v \in B) = (1 - p)^{d_v + 1}$, as we need that neither $v$ nor any of its $d_v$ neighbors belong to $A$. Thus, the expected size of $B$ is $\sum_{v \in V} (1 - p)^{d_v + 1}$. Since $(1-p)^x$ is a convex function, and the average degree is at most $d$, we have
\[ \mathbb{E}[|B|] = \sum_{v \in V} (1-p)^{d_v + 1} \ge n (1-p)^{d + 1} = n \left( 1 - \frac{\log d}{2d} \right)^{d+1}. \]
Computation shows that this is at least $\frac{3n}{4 d^{1/2}}$. Since $|B|$ cannot be larger than $n$, we have
\[ \frac{3n}{4d^{1/2}} \le \mathbb{E}[|B|] \le n \mathbb{P} \left( |B| \ge \frac{n}{2 d^{1/2}} \right) + \frac{n}{2d^{1/2}} \mathbb{P} \left( |B| \le \frac{n}{2d^{1/2}} \right) \le n \mathbb{P}\left( |B| \ge \frac{n}{2d^{1/2}} \right) + \frac{n}{2d^{1/2}}, \]
whence it follows that $\mathbb{P} \left( |B| \ge \frac{n}{2d^{1/2}} \right) \ge \frac{1}{4d^{1/2}} > \frac{1}{d^3}$.

Hence, with positive probability, we have both $|A| \ge \frac{n \log d}{4d}$ and $|B| \ge \frac{n}{2 d^{1/2}} \ge \frac{n \log d}{4 d}$, and the existence of such a pair of sets shows $z(G) \ge \frac{n \log d}{4d}$, as required.

\medskip

For part (ii), we need to show the existence of a graph $G$, of average degree at most $d$, for which $z(G) \le \frac{2 n \log d}{d}$. Note that we always have $z(G) \le \frac{n}{2}$, and so this is trivial if $d \le 8$. 

Now consider the random graph $G \left( n,\frac{d}{n} \right)$, where every edge is present independently with probability $p = \frac{d}{n}$. The number of edges is a binomial random variable, whose median is at most $\left \lceil \frac{d(n-1)}{2} \right \rceil$, and hence $\mathbb{P}\left(e(G) > \frac{dn}{2} \right) < \frac12$.

We can use a straightforward union bound to show that there are no large subsets without any edges between them. Indeed, the expected number of pairs of sets of $z$ vertices such that all $z^2$ cross-edges are missing is at most 
\[ \binom{n}{z}^2 (1-p)^{z^2} \le \left( \frac{ne}{z} \right)^{2z} e^{-pz^2} = \left( \frac{n^2 e^2 }{z^2 e^{pz} } \right)^z. \]

Substituting our choice of $z = \frac{2 n \log d}{d}$, this simplifies to $\left( \frac{e^2}{4 \log^2 d} \right)^{\frac{2n \log d}{d}}$. Since $d \ge 8$, this is at most $2^{-\frac{2n \log d}{d}}$, which is less than $\frac12$.

Thus, with positive probability, $G$ is such that $e(G) \le \frac{dn}{2}$ and $z(G) \le \frac{2n \log d}{d}$, as required.
\end{proof}

\begin{remark}
We have dealt with graphs of bounded average degree for simplicity, so that we could use the binomial random graph in part (ii). If one is primarily interested in $d$-regular graphs, as we have been using in this paper, then the upper bound in part (i) naturally still applies. For the lower bound in part (ii), one must replace the binomial random graph with the random $d$-regular graph. At the expense of more complicated calculations, a similar bound can be shown, provided $d$ is not too small.
\end{remark}

\end{document}